\documentclass[12pt]{amsart}

\usepackage{amssymb}

\usepackage{enumerate}

\usepackage{graphicx}

\makeatletter
\@namedef{subjclassname@2010}{%
  \textup{2010} Mathematics Subject Classification}
\makeatother

\newtheorem{thm}{Theorem}[section]

\newtheorem{prob}[thm]{Problem}
\newtheorem{con}{Conjecture}

\newtheorem{mainthm}[thm]{Main Theorem}

\theoremstyle{definition}

\newtheorem{rem}[thm]{Remark}

\numberwithin{equation}{section}

\newcommand{\bs}{\backslash}
\newcommand{\N}{\mathbb{N}}
\newcommand{\Z}{\mathbb{Z}}
\newcommand{\Q}{\mathbb{Q}}

\begin{document}


\baselineskip=17pt



\title[The behavior of the $p$-adic valuations of Stirling numbers]{A note on $p$-adic locally analytic functions with application to behavior of the $p$-adic valuations of Stirling numbers}

\author[P. Miska]{Piotr Miska}
\address{Institute of Mathematics \\
	Faculty of Mathematics and Computer Science \\
	Jagiellonian University in Cracow\\
	ul. {\L}ojasiewicza 6, 30-348, Krak\'{o}w, Poland}
\email{piotr.miska@student.uj.edu.pl}

\date{}

\begin{abstract}
The aim of this paper is to prove conjectures concerning $p$-adic valuations of Stirling numbers of the second kind $S(n,k)$, $n,k\in\N_+$, stated by Amdeberhan, Manna and Moll and Berrizbeitia et al., where $p$ is a prime number. The proof is based on elementary facts from $p$-adic analysis.
\end{abstract}



\maketitle

\section{Introduction}

At the beginning of the paper, we denote the set of non-negative integers (i. e. positive integers with $0$ included) by $\N$ while the set of positive integers (with $0$ excluded) by $\N_+$.

Let $n,k\in\N_+$. The Stirling number of the second kind $S(n,k)$ counts the number of partitions of a set with $n$ elements into exactly $k$ nonempty subsets. Using a simple combinatorial reasoning one can obtain the recurrence
\begin{equation*}
S(n+1,k+1)=S(n,k)+(k+1)S(n,k+1),\quad\quad n,k\in\N_+,
\end{equation*}
with initial conditions $S(1,1)=1$, $S(1,k)=0$ for $k>1$ and $S(n,1)=1$ for $n\in\N_+$. The exact formula for the Stirling numbers of the second kind takes the form
\begin{equation*}
S(n,k)=\frac{1}{k!}\sum_{j=1}^k (-1)^{k-j}{k\choose j}j^n.
\end{equation*}
It is worth to note that the numbers $k!S(n,k)$ also have a combinatorial interpretation. Namely, $k!S(n,k)$ is the number of all surjections of a set with $n$ elements onto a set with $k$ elements.

Let us recall that for a prime number $p$ and a nonzero rational number $x$ we define the $p$-adic valuation $v_p(x)$ of the number $x$ as a number $t\in\Z$ such that $x=\frac{a}{b}p^t$, where $a\in\mathbb{Z}$, $b\in\mathbb{N}_+$, $\gcd(a,b)=1$ and $p\nmid ab$. Moreover, we put $v_p(0)=+\infty$. The problem of computation of the $p$-adic valuations (with emphasis on $2$-adic valuations) of Stirling numbers of the second kind and their relatives generated a lot of research, e.g. \cite{Cla, Dav2, Dav3, HZZ, Leng2, Lun1, Wan}. The Stirling numbers of the second kind appear in algebraic topology. Their connections with algebraic topological problems can be found in \cite{BenDav, CraKna, Dav4, DavPot, DavSun, KomLam, Lun2}. Amdeberhan, Manna and Moll in \cite{AMM} and Berrizbeitia, Medina, Moll, Moll and Noble in \cite{BMMMN} stated two conjectures on the behavior of the $p$-adic valuations of the numbers $S(n,k)$. In order to formulate these conjectures we will denote by $[a]_d$ the congruence class of an integer $a$ modulo a positive integer $d$:
\begin{equation*}
[a]_d=\{n\in\Z: n\equiv a\pmod{d}\}.
\end{equation*}
Moreover, for a sequence $(c(n))_{n\in\N_+}$ of rational numbers we will write
\begin{equation*}
c\left([a]_d\right)=\left\{c(n): n\in [a]_d, n\in\N_+\right\}.
\end{equation*}
In particular,
\begin{equation*}
S\left([a]_d,k\right)=\left\{S(n,k): n\in [a]_d, n\in\N_+\right\}
\end{equation*}
for $k\in\N$. Consequently,
\begin{equation*}
v_p\left(c\left([a]_d\right)\right)=\left\{v_p(c(n)): n\in [a]_d, n\in\N_+\right\}
\end{equation*}
and
\begin{equation*}
v_p\left(S\left([a]_d,k\right)\right)=\left\{v_p(S(n,k)): n\in [a]_d, n\in\N_+\right\}
\end{equation*}
for a prime number $p$. Fixing a sequence $(c(n))_{n\in\N_+}$ of rational numbers and a prime number $p$ we say that a congruence class $[a]_d$ is \emph{constant} with $p$-adic valuation $t$ if $v_p\left(c\left(n\right)\right)=t$ for each $n\in [a]_d$. Every congruence class that is not constant will be called \emph{non-constant}. The value $\min v_p\left(c\left([a]_d\right)\right)$ (if exists) will be called \emph{the least $p$-adic valuation of the class $[a]_d$}. Then the first conjecture, concerning the $2$-adic valuations of Stirling numbers of the second kind, can be presented in the following way.

\begin{con}[Conjecture 1.6. in \cite{AMM}]\label{Con1}
Fix a number $k\in\N_+$. Then there exist $m_0(k)\in\N_+$ and $\mu(k)\in\N$ such that for any integer $m\geq m_0(k)$ there are exactly $\mu(k)$ non-constant congruence classes $[a]_{2^m}$ modulo $2^m$ and each non-constant class modulo $2^m$ splits into one constant and one non-constant class modulo $2^{m+1}$.
\end{con}

The authors of \cite{AMM} proved the above conjecture for $k\leq 5$. Their work was continued by Bennett and Mosteig in \cite{BenMos}, who proved this conjecture in case $5\leq k\leq 20$. Conjecture \ref{Con1} in a slightly different form was proved by Davis in \cite{Dav1} for partial Stirling numbers of the second kind $S_2(n,k)=\frac{1}{k!}\sum_{j=1, 2\nmid j}^k (-1)^{k-j}{k\choose j}j^n$, where $k\leq 36$. It is remarkable that according to Conjecture \ref{Con1} all the numbers $S(2^n,k)$, $n\gg 0$, have the same $2$-adic valuation. The independence of $2$-adic valuation of $S(2^n,k)$ on $n$ was noticed by Lengyel in \cite{Leng}. He conjectured and proved in some cases that
\begin{equation}\label{LengWan}
v_2(k!S(2^n,k))=k-1
\end{equation}
for every non-negative integer $n$ such that $2^n\geq k$. The above equality in full generality was proved by De Wannemacker in \cite{Wan}.

For an odd prime number $p$ and positive integer $m$ we put
$$L_{p^m}=(p-1)p^{\lceil\log_p k\rceil +m-2}.$$
Then the sequence $(S(n,k)\pmod{p^m})_{n\in\N_{\geq k}}$ is ultimately periodic with period $L_{p^m}$ (see \cite{BeckRio, Car}). Now we are ready to present the second conjecture, which generalizes the first one for arbitrary prime number $p$.

\begin{con}[Conjecture 5.4. in \cite{BMMMN}]\label{Con2}
Fix a positive integer $k$ and a prime number $p$ such that $k<p$. Then there exist $m_0(k)\in\N_+$ and $\mu(k)\in\N$ such that for any integer $m\geq m_0(k)$ there are exactly $\mu(k)$ non-constant congruence classes $[a]_{L_{p^m}}$ modulo $L_{p^m}$ and each non-constant class modulo $L_{p^m}$ splits into $p-1$ constant and one non-constant class modulo $L_{p^{m+1}}$.
\end{con}

The above conjecture implies that, under certain conditions, for any $a\in\N$ and sufficiently large $n\in\N$ the $p$-adic valuation of the number $S(ap^n(p-1),k)$ is independent on $a$ and $n$. This fact was proved by Gessel and Lengyel in \cite{GesLeng} provided that $\frac{k}{p}$ is not an odd integer. More precisely, they gave a formula 
\begin{equation}\label{GesLeng}
v_p(k!S(ap^n(p-1),k))=\left\lfloor\frac{k-1}{p-1}\right\rfloor+\tau_p(k),
\end{equation}
where $\tau_p(k)\in\N$. It is not clear in general what values $\tau_p(k)$ can take. The formula for $\tau_p(k)$ is given for $p\in\{3,5\}$. Moreover, $\tau(k)=0$ for $k$ being a multiple of $p-1$. As we can see, the formula (\ref{GesLeng}) is a generalization of (\ref{LengWan}).

We will say that a congruence class $[a]_d$ of an integer $a$ modulo a positive integer $d$ is \emph{almost constant} with $p$-adic valuation $t$ if $v_p(S(n,k))=t$ for almost all positive integers $n\in [a]_d$. The congruence class which is not almost constant will be called \emph{essentially non-constant}. If the limit
$$\liminf_{n\rightarrow +\infty, n\in [a]_d}  v_p(c_n)=:\liminf v_p(S([a]_d,k))$$
exists then it will be called \emph{the ultimate minimal $p$-adic valuation of the class $[a]_d$}. Now we are ready to state the theorem describing the behavior of the $p$-adic valuations of the Stirling numbers of the second kind.

\begin{mainthm}\label{thm1}
Fix a positive integer $k$ and a prime number $p$. Then there exist $m_0(k)\in\N_+$ and $\mu(k)\in\N$ such that for any integer $m\geq m_0(k)$ there are exactly $\mu(k)$ essentially non-constant congruence classes $[a]_{p^{m-1}(p-1)}$ modulo $p^{m-1}(p-1)$. Moreover, each essentially non-constant class $[a]_{p^{m-1}(p-1)}$ modulo $p^{m-1}(p-1)$ splits into $p-1$ almost constant classes $[a_1]_{p^{m}(p-1)}$, ..., $[a_{p-1}]_{p^{m}(p-1)}$ modulo $p^{m}(p-1)$ such that
$$v_p(S([a_t]_{p^{m}(p-1)}))=\{\min v_p(S([a]_{p^{m-1}(p-1)}))\},\quad t\in\{1,...,p-1\},$$
and one essentially non-constant class $[a_0]_{p^{m}(p-1)}$ modulo $p^{m}(p-1)$ with
\begin{align*}
0<\min v_p(S([a_0]_{p^{m}(p-1)}))-\min v_p(S([a]_{p^{m-1}(p-1)}))<k-\left\lfloor\frac{k}{p}\right\rfloor.
\end{align*}

More precisely, for each essentially non-costant congruence class\linebreak $[a]_{p^{m_0-1}(p-1)}$ modulo $p^{m_0-1}(p-1)$ there are $\alpha\in\Z$ and $l\in\left\{1,...,k-\left\lfloor\frac{k}{p}\right\rfloor-1\right\}$ such that for each $m\geq m_0$ the ultimate minimal $p$-adic valuation of the unique essentially non-constant class modulo $p^{m-1}(p-1)$ contained in\linebreak $[a]_{p^{m_0-1}(p-1)}$ is equal to $\alpha+l(m-m_0)$.

In case of $p>k$ in the above statement we can omit words ``almost'', ``essentially'' and ``ultimate''. Then, for each non-costant congruence class $[a]_{p^{m_0-1}(p-1)}$ modulo $p^{m_0-1}(p-1)$ there are $\beta\in\Z$, $l\in\left\{1,...,k-\left\lfloor\frac{k}{p}\right\rfloor-1\right\}$ and $x_0\in\Z_p$ such that $v_p(S(n,k))=\beta+lv_p(n-x_0)$ for $n\in [a]_{p^{m_0-1}(p-1)}$.
\end{mainthm}

\begin{rem}
{\rm In cases of $k=1$ or $k=p=2$ one can straightforward check $v_p(S(n,k))=0$ for all $n\in\N_+$. Hence the statement of Theorem \ref{thm1} remains true in these cases, although the set $\left\{t+1,...,t+k-\left\lfloor\frac{k}{p}\right\rfloor-1\right\}$ is empty. In fact, we then have no non-constant classes with respect to the sequence $(S(n,k))_{n\in\N_+}$.}
\end{rem}



\section{Elementary $p$-adic analytic approach}\label{sec1}

In order to prove Theorem \ref{thm1} we will use some basic facts about the field of $p$-adic numbers and $p$-adic locally analytic functions.

First, let us recall the construction of the field of $p$-adic numbers. For every rational number $x$ we define its $p$-adic norm $|x|_p$ by the formula
\begin{equation*}
|x|_p =
\begin{cases}
p^{-v_p(x)}, & \mbox{when } x\neq 0
\\ 0, & \mbox{when } x=0
\end{cases}.
\end{equation*}
Since for all rational numbers $x,y$ we have $|x+y|_p \leq \min\{ |x|_p, |y|_p\}$, hence $p$-adic norm gives a metric space structure on $\Q$. Namely, the distance between rational numbers $x,y$ is equal to $d_p(x,y) = |x-y|_p$. The field $\Q$ equipped with the $p$-adic metric $d_p$ is not a complete metric space. The completion of $\Q$ with respect to this metric has a structure of a field and this field is called the field of $p$-adic numbers and denoted by $\Q_p$. We extend the $p$-adic valuation and $p$-adic norm on $\Q_p$ in the following way: $v_p(x) = \lim_{n\rightarrow +\infty} v_p(x_n)$, $|x|_p = \lim_{n\rightarrow +\infty} |x_n|_p$, where $x\in\Q_p$, $(x_n)_{n\in\N} \in \Q^{\N}$ and $x = \lim_{n\rightarrow +\infty} x_n$. The values $v_p(x)$ and $|x|_p$ do not depend on the choice of a sequence $(x_n)_{n\in\N}$, thus they are well defined. For the proofs of these facts one can consult \cite{Mahl}.

We define the ring of integer $p$-adic numbers $\Z_p$ as a set of all $p$-adic numbers with non-negative $p$-adic valuation. Note that $\Z_p$ is the completion of $\Z$ as a space with the $p$-adic metric. In the sequel we will use the fact that $\Z_p$ is a compact metric space.

We assume that the expression $x \equiv y \pmod{p^k}$ means $v_p(x-y) \geq k$ for prime number $p$, an integer $k$ and $p$-adic numbers $x,y$.

By $p$-adic continuous function we mean a function $f:S\rightarrow\Q_p$ defined on some subset $S$ of $\Q_p$, which is continuous with respect to the $p$-adic metric. Assuming that $S$ is an open subset of $\Q_p$, we will say that $f$ is differentiable at a point $x_0\in S$, if there exists a limit $\lim_{x\rightarrow x_0}\frac{f(x)-f(x_0)}{x-x_0}$. In this situation this limit will be called the derivative of $f$ at the point $x_0$ and denoted  by $f'(x_0)$. If $f$ is differentiable at each point of its domain then we will say that $f$ is a $p$-adic differentiable function and the mapping $S\ni x\mapsto f'(x)\in\Q_p$ is the derivative of $f$. Recursively we are able to define the higher order derivatives of $f$ as follows: $f^{(1)}(x)=f'(x)$ and if $f^{(n)}$ is a $p$-adic differentiable function for $n\in\N_+$ then we define $f^{(n+1)}(x)=\left(f^{(n)}\right)'(x)$. Finally, still assuming $S$ to be an open subset of $\Q_p$, we will say that $f:S\rightarrow\Q_p$ is a $p$-adic locally analytic function if for each $x_0\in S$ there exists a neighbourhood $U\subset S$ of $x_0$ and the sequence $(a_n)_{n\in\N}$ of $p$-adic numbers such that for any $x\in U$ the series $\sum_{n=0}^{+\infty} a_n(x-x_0)^n$ is convergent and its sum is equal to $f(x)$. A $p$-adic locally analytic function $f$ is continuous, arbitrarily many times differentiable and $f^{(n)}(x_0)=n!a_n$ for $n\in\N$ (see \cite{Mahl}).


If $f:\Z_p\rightarrow\Q_p$ is a $p$-adic locally analytic function then we can easily obtain a description of the $p$-adic valuations of numbers $f(n)$, $n\in\N$.

\begin{thm}\label{thm2}
Let $f:\Z_p\rightarrow\Q_p$ be a $p$-adic locally analytic function. Then there exist $m_0\in\N_+$ and $\mu\in\N$ such that for any integer $m\geq m_0$ there are exactly $\mu$ non-constant congruence classes $[a]_{p^m}$ modulo $p^m$ with respect to the sequence $(f(n))_{n\in\N_+}$ and each non-constant class modulo $p^m$ splits into $p-1$ constant classes $[a_1]_{p^{m+1}}$, ..., $[a_{p-1}]_{p^{m+1}}$ modulo $p^{m+1}$ with $p$-adic valuation equal to $\min v_p(f([a]_{p^m}))$ and one non-constant class $[a_0]_{p^{m+1}}$ modulo $p^{m+1}$ with
$$\min v_p(f([a_0]_{p^{m+1}}))>\min v_p(f([a]_{p^m})).$$

More precisely, for each non-costant congruence class $[a]_{p^{m_0}}$ modulo $p^{m_0}$ there are $\beta\in\Z$, $l\in\N_+$ and $x_0\in\Z_p$ such that $v_p(f(n))=\beta+lv_p(n-x_0)$ for $n\in [a]_{p^{m_0}}$. In particular, for each non-costant congruence class $[a]_{p^{m_0}}$ modulo $p^{m_0}$ there are $\alpha\in\Z$ and $l\in\N_+$ such that for each $m\geq m_0$ the minimal $p$-adic valuation of the unique non-constant class modulo $p^m$ contained in $[a]_{p^{m_0}}$ is equal to $\alpha+l(m-m_0)$.
\end{thm}

\begin{proof}
Choose an arbitrary $x_0\in\Z_p$. We consider three cases.

\noindent\textbf{Case 1.} If $f(x_0)\neq 0$ then $v_p(f(x_0))<+\infty$. By continuity of $f$ there exists a neighbourhood $U_{x_0}\subset\Z_p$ of $x_0$ such that $v_p(f(x)-f(x_0))>v_p(f(x_0))$ for each $x\in U_{x_0}$. As a consequence, we have $v_p(f(x))=v_p(f(x_0))$ for $x\in U_{x_0}$.

\noindent\textbf{Case 2.} There exists a neighbourhood $U_{x_0}\subset\Z_p$ of $x_0$ such that $f(x)=0$ for each $x\in U_{x_0}$ and then $v_p(f(x))=+\infty$.

\noindent\textbf{Case 3.} There holds $f(x_0)=0$ and $f$ is not identically equal to zero function on any neighbourhood of $x_0$. Then for some neighbourhood $U_{x_0}\subset\Z_p$ of $x_0$ we can write $f(x)=\sum_{n=l}^{+\infty} a_n(x-x_0)^n$ for $x\in U_{x_0}$, where $l>0$ and $a_l\neq 0$. Since $$\lim_{x\rightarrow x_0}\sum_{n=l}^{+\infty} a_n(x-x_0)^{n-l}=a_l,$$ we may assume without loss of generality that the neighbourhood $U_{x_0}$ is so small that $v_p\left(\sum_{n=l}^{+\infty} a_n(x-x_0)^{n-l}\right)=v_p(a_l)$ for each $x\in U_{x_0}$. Then
\begin{align*}
v_p(f(x))&=v_p\left(\sum_{n=l}^{+\infty} a_n(x-x_0)^n\right)=v_p\left(\sum_{n=l}^{+\infty} a_n(x-x_0)^{n-l}\right)+v_p\left((x-x_0)^l\right)\\
&=v_p(a_l)+lv_p(x-x_0).
\end{align*}

For each $x_0\in\Z_p$ we take $U_{x_0}$ specified in the above cases. We may choose $U_{x_0}$ to be a ball in the $p$-adic metric:
\begin{align*}
B(x_0,p^{-m})&=\{x\in\Z_p: |x-x_0|_p\leq p^{-m}\}=\{x\in\Z_p: v_p(x-x_0)\geq m\}\\
&=\{x\in\Z_p: x\equiv x_0\pmod{p^m}\},
\end{align*}
where $m\in\N$. Let us notice that $B(x_0,p^{-m})\cap\Z=[x_1]_{p^m}$ for any integer $x_1$ satisfying $x_1\equiv x_0\pmod{p^m}$. The family $\{U_{x_0}\}_{x_0\in\Z_p}$ is an open cover of $\Z_p$. By compactness of $\Z_p$ we may choose $x_1,...,x_r\in\Z_p$ such that $\{U_{x_j}\}_{j=1}^r$ is\linebreak a cover of $\Z_p$. We write $U_{x_j}=B(x_j,p^{-m_j})$ for $j\in\{1,...,r\}$. Every two balls in the $p$-adic metric are either disjoint or one contains the second one. Hence we can choose a finite subcover $(U_{x_j})_{j=1}^r$ to be consisting of pairwise disjoint balls. We put $m_0=\max_{1\leq j\leq r}m_j$. Let us consider the balls $B(i,p^{-m_0})$, $i\in\{0,...,p^{m_0}-1\}$. They are a partition of $\Z_p$. Fixing $i\in\{0,...,p^{m_0}-1\}$ there exists $j\in\{1,...,r\}$ such that $B(i,p^{-m_0})\subset B(x_j,p^{-m_j})$. If $x_j$ is as in case Case 1 or Case 2 in the above reasoning then $\#\{v_p(f(x)): x\in B(i,p^{-m_0})\}=1$ and, as a result, the congruence class $[i]_{p^{m_0}}$ is constant.

If $x_j$ is as in Case 3 and $v_p(i-x_j)<m_0$ then for $x\in B(i,p^{-m_0})$ we have $f(x)=\sum_{n=l}^{+\infty} a_n(x-x_j)^n$ and
\begin{align*}
v_p(f(x))&=v_p(a_l)+lv_p(x-x_j)=v_p(a_l)+lv_p((x-i)+(i-x_j))\\
&=v_p(a_l)+lv_p(i-x_j).
\end{align*}
Thus the congruence class $[i]_{p^{m_0}}$ is constant. On the other hand, if $x_j$ is as in Case 3 and $v_p(i-x_j)\geq m_0$ then $v_p(f(x))=v_p(a_l)+lv_p(x-x_j)$ and we easily conclude that for each $m\geq m_0$ there exists exactly one non-constant class $[h]_{p^m}$ (namely the class of $h\equiv x_j\pmod{p^m}$) contained in $[i]_{p^{m_0}}$. This class splits into $p-1$ constant classes modulo $p^{m+1}$ and one non-constant class modulo $p^{m+1}$. Each non-constant class modulo $p^m$, $m\geq m_0$, corresponds to one $j\in\{1,...,r\}$ such that $x_j$ is as in Case 3. Hence the number $\mu$ postulated in the statement of our theorem is the number of indices $j$ for which $x_j$ is as in Case 3. In other words, $\mu$ is the number of the isolated zeros of the function $f$.
\end{proof}

The reasoning in the proof of Theorem \ref{thm2} recalls us that the set of zeros of a $p$-adic locally analytic function $f:\Z_p\rightarrow\Q_p$ is a union of a finite subset and a clopen subset of $\Z_p$. This fact is used in a proof of basic version of Skolem-Mahler-Lech theorem that if $(c(n))_{n\in\N_+}$ is a sequence of rational numbers given by a linear recurrence $c(n+k)=\alpha_0c(n)+...+\alpha_{k-1}c(n+k-1)$ for some $\alpha_0,...,\alpha_{k-1}\in\Q$ then the set $\{n\in\N_+: c(n)=0\}$ is a union of\linebreak a finite set and finitely many arithmetic progressions (see \cite{Sko}).

The next fact that we utilize to prove the main result of the paper is a generalization of the lifting the exponent lemma. This generalization describes the behavior of the $p$-adic valuations of a linear combination of exponential functions. Let us see that if $p$ is an odd number and $a$ is a $p$-adic integer congruent to $1$ modulo $p$, i.e., $v_p(a-1)\geq 1$, then writing $a=1+pb$ for some $b\in\Z_p$ we can expand the expression $a^x$ for $x\in\N$ as follows:
\begin{equation}\label{exp}
a^x=(1+pb)^x=\sum_{j=0}^x {x\choose j}p^jb^j=\sum_{j=0}^{+\infty} b^j(x)_j\frac{p^j}{j!},
\end{equation}
where $(x)_j=x(x-1)\cdot...\cdot(x-j+1)$ for $j\in\N_+$ and $(x)_0=1$ means Pochhammer symbol. The series on the right-hand side of the equation (\ref{exp}) is convergent for every $x\in\Z_p$. Hence we are able to extend the exponential function $a^x$ for $x\in\Z_p$ defining it as $\sum_{j=0}^{+\infty} b^j(x)_j\frac{p^j}{j!}$. The function $a^x$, $x\in\Z_p$, defined in a such way, is a $p$-adic locally analytic function (see \cite{Mahl}). In particular, this function is continuous. This fact combined with the identity $a^{x+y}=a^xa^y$ for $x,y\in\N$ and the density of $\N$ in $\Z_p$ implies the identity $a^{x+y}=a^xa^y$ for $x,y\in\Z_p$.

The above preparation becomes to be true for $p=2$ if we assume that $v_2(x)\geq 1$.

Assuming that $p$ is a prime number and $a$ is a $p$-adic integer of the form $a=1+pb$ for some $b\in\Z_p$, we define the $p$-adic logarithm of the number $a$ as follows:
\begin{align*}
\log_p a=\sum_{j=1}^{+\infty} \frac{(-1)^{j-1}}{j}(a-1)^j=\sum_{j=1}^{+\infty} \frac{(-1)^{j-1}b^jp^j}{j}.
\end{align*}
One can see that the series defining the $p$-adic logarithm is convergent. Hence the function $\log_p x$ is a $p$-adic locally analytic function defined on the closed ball $B(1,p^{-1})$. We have $\log_p a_1+\log_p a_2=\log_p a_1a_2$ for any $a_1,a_2\in B(1,p^{-1})$ and the derivative of the function $a^x$ (of variable $x\in\Z_p$) is equal to $a^x\log_p a$ for $a\in B(1,p^{-1})$ (see \cite{Mahl}). Moreover, $\log_p a=0$ if and only if $a$ is a root of unity (see \cite{Schi}). Since $a\equiv 1\pmod{p}$, we conclude that $\log_p a=0$ only for $a=1$ or $p=2$ and $a=-1$.

Now we are ready to state and prove the announced generalization of the lifting the exponent lemma.

\begin{thm}\label{thm3}
Let $p$ be a prime number and $k\in\N_+$, $k\geq 2$. Assume that $a_1,...,a_k$ are pairwise distinct $p$-adic integers congruent to $1$ modulo $p$. In case of $p=2$ we assume additionally that $\frac{a_j}{a_i}\neq -1$ for any $i,j\in\{1,...,k\}$. Let $c_1,...,c_k\in\Q_p\bs\{0\}$. Assume that a number $x_0\in\Z_p$ is a zero of the function $f(x)=\sum_{i=1}^k c_ia_i^x$, $x\in\Z_p$. Then there exist an integer $\alpha$,\linebreak a positive integer $l<k$ and a non-negative integer $m$ such that for any $t\in\Z_p$ congruent to $x_0$ modulo $p^m$ we have $v_p(f(t))=\alpha+lv_p(t-x_0)$. In particular, $x_0$ is an isolated zero of the function $f$.
\end{thm}

\begin{proof}
From the previous theorem we know that there exist $\alpha\in\Z$, $l\in\N\cup\{+\infty\}$ and $m\in\N$ such that $v_p(f(t))=\alpha+lv_p(t-x_0)$ for any $t\in B(x_0,p^{-m})$, where $l=\inf\{n\in\N: f^{(n)}(x_0)\neq 0\}$ (we set that $\inf\varnothing=+\infty$). Hence it remains to prove that $l<k$. If we assume the contrary, then $f^{(n)}(x_0)=0$ for any $n\in\{0,...,k-1\}$. On the other hand,
\begin{align*}
f^{(n)}(x_0)=\sum_{i=1}^k c_ia_i^{x_0}(\log_p a_i)^n,\quad n\in\N.
\end{align*}
Then the $k$-tuple $(c_ia_i^{x_0})_{i=1}^k$ is a non-zero solution of the system of linear equations $\sum_{i=1}^k (\log_p a_i)^nX_i=0$, $0\leq n\leq k-1$, with unknowns $X_i$, $1\leq i\leq k$. However, the matrix of this system is the Vandermonde matrix for values $\log_p a_i$, $1\leq i\leq k$. Hence its determinant is equal to $\prod_{1\leq i<j\leq k} (\log_p a_j-\log_p a_i)\neq 0$, since the values $a_i$, $1\leq i\leq k$, are pairwise distinct and no ratio of the form $\frac{a_j}{a_i}\neq -1$. Thus the only solution of this system is zero. The contradiction shows the inequality $l<k$.
\end{proof}

\begin{rem}
{\rm It is worth to see that in general the value $l$ in the statement of Theorem \ref{thm3} can be greater than $1$. Let us take any two distinct integers $a$, $b$ congruent to $1$ modulo $p$ if $p>2$, or congruent to $1$ modulo $4$ if $p=2$. Then the function $f(x)=(a^2)^x+(b^2)^x-2(ab)^x$ vanishes at $0$ and $f'(0)=\log_p a^2 +\log_p b^2 -2\log_p ab=0$. Thus $l>1$. On the other hand, by Theorem \ref{thm3} we have $l<3$. As a result, $l=2$ and finally there exist $\alpha\in\Z$ and $m\in\N$ that $v_p(f(x))=\alpha+2v_p(x)$ for $x\in B(0,p^{-m})$.

In fact, if we fix $k$, then $l$ may attain every value from the set $\{0,...,k-1\}$. It suffices to choose $c_1,...,c_k\in\Q_p$ such that $\sum_{i=1}^k (\log_p a_i)^nc_i=0$, $0\leq n\leq l-1$, and $\sum_{i=1}^k (\log_p a_i)^lc_i\neq 0$. We are able to choose such $c_1,...,c_k$ since, as we noted in the proof of Theorem \ref{thm3}, the determinant of the system of equations $\sum_{i=1}^k (\log_p a_i)^nX_i=0$, $0\leq n\leq k-1$, is non-zero. Then for some integer $\alpha$ we have $v_p(f(x))=\alpha+lv_p(x)$ for any $x$ sufficiently close to $0$ with respect to the $p$-adic metric.

Unfortunately, fixing $k>l>2$ we do not know if we can choose $a_1,...,a_k$ to be integers and $c_1,...,c_k$ to be rational numbers. Hence, in case $a_1,...,a_k\in\Z$, $c_1,...,c_k\in\Q$ we do not know about the value $l$ anymore than from Theorem \ref{thm3}.}
\end{rem}

\section{Proof of Theorem \ref{thm1}}

Let us denote, as in \cite{Cla}, $$T_p(n,k)=\sum_{j=1, p\nmid j}^k (-1)^{k-j}{k\choose j}j^n.$$ From the preparation before Theorem \ref{thm3}, if $j\in\Z$ is not divisible by $p$ then for $a\in\{0,...,p-2\}$ the function $\N_+\ni n\mapsto j^{a+n(p-1)}\in\Z$ extends to a $p$-adic locally analytic function defined on $\Z_p$. Hence for each $a\in\{0,...,p-2\}$ the function $\N_+\ni n\mapsto T_p(a+n(p-1),k)\in\N$ is a restriction of a $p$-adic locally analytic function $f_{a,k}:\Z_p\rightarrow\Z_p$ to $\N_+$. Thus we apply Theorem \ref{thm2} and Theorem \ref{thm3} to each function $f_{a,k}$ and we obtain $m_{0,a}$ and $\mu_a$ as specified in the statement of Theorem \ref{thm2}. We thus put $$m_0=m_0(k)=1+\max_{0\leq a\leq p-2} m_{0,a}$$ and $$\mu=\mu(k)=\sum_{a=0}^{p-2} \mu_a.$$ Then there holds the statement of Theorem \ref{thm1} for the sequence $T_p(n,k)$ with the words ``almost'' and ``essentially'' ommited.

If $p>k$ then $S(n,k)=\frac{1}{k!}T_p(n,k)$ and the statement holds with omitted words ``almost'' and ``essentially''. Let us fix a non-constant class $[a]_{p^{m_0-1}(p-1)}$ modulo $p^{m_0-1}(p-1)$. Each member of this class can be written in the form $n=a_0+\tilde{n}(p-1)$, where $a_0=a\bmod{p-1}$. Then, by Theorem \ref{thm2} and Theorem \ref{thm3}, we have $\beta\in\Z$, $l\in\{1,...,k-1\}$ and $\tilde{x}_0\in\Z_p$ such that
\begin{align*}
v_p(S(n,k))=v_p(S(a_0+\tilde{n}(p-1),k))=v_p(f_{a_0,k}(\tilde{n}))=\beta+lv_p(\tilde{n}-\tilde{x}_0)
\end{align*}
for $n\in [a]_{p^{m_0-1}(p-1)}$. Putting $x_0=a_0+\tilde{x}_0(p-1)$, since $v_p(n-x_0)=v_p(\tilde{n}-\tilde{x}_0)$, we get
\begin{align*}
v_p(S(n,k))=\beta+lv_p(n-x_0)
\end{align*}
for $n\in [a]_{p^{m_0-1}(p-1)}$.

In general
\begin{align}\label{sum}
S(n,k)=\frac{1}{k!}\left(T_p(n,k)+\sum_{0<j\leq k, p\mid j}(-1)^{k-j}{k\choose j}j^n\right).
\end{align}
Certainly, $v_p\left(\sum_{0<j\leq k, p\mid j}(-1)^{k-j}{k\choose j}j^n\right)\geq n$. Hence, if $[a]_{(p-1)p^{m-1}}$ is a constant class with respect to $(T_p(n,k))_{n\in\N_+}$ with the $p$-adic valuation $t$ then, by (\ref{sum}), $v_p(S(n,k))=t-v_p(k!)$ for each $n\in [a]_{(p-1)p^{m-1}}$ satisfying $n>t$. On the other hand, if $[a]_{(p-1)p^{m-1}}$ is a non-constant class with respect to $(T_p(n,k))_{n\in\N_+}$ then it contains infinitely many constant classes with respect to $(T_p(n,k))_{n\in\N_+}$ with pairwise distinct $p$-adic valuations. Hence the numbers $v_p(S(n,k))$, $n\in [a]_{(p-1)p^{m-1}}$, attain infinitely many values and the class $[a]_{(p-1)p^{m-1}}$ is not almost constant with respect to $(S(n,k))_{n\in\N_+}$.

\section{Concluding remarks}

Theorem \ref{thm1} states that for a given prime number $p$ and a positive integer $k$ there exists $m_0=m_0(k)\in\N_+$ such that for each essentially non-costant congruence class $[a]_{p^{m_0-1}(p-1)}$ modulo $p^{m_0-1}(p-1)$ with respect to the sequence $(S(n,k))_{n\in\N_+}$ and the prime number $p$ there are $\alpha\in\Z$ and $l\in\left\{1,...,k-\left\lfloor\frac{k}{p}\right\rfloor-1\right\}$ such that for each $m\geq m_0$ the ultimate minimal $p$-adic valuation of the unique essentially non-constant class modulo $p^{m-1}(p-1)$ contained in $[a]_{p^{m_0-1}(p-1)}$ is equal to $\alpha+l(m-m_0)$. However, we do not know any example of a prime number $p$, a positive integer $k$ and an essentially non-constant class modulo $p^{m_0-1}(p-1)$ with respect to the sequence $(S(n,k))_{n\in\N_+}$ and the prime number $p$ for which $l>1$.

\begin{prob}
Do there exist a prime number $p$ and positive integers $k,a$ such that $[a]_{p^{m_0-1}(p-1)}$ ($m_0$ as specified in the statement of Theorem \ref{thm1}) is an essentially non-constant class modulo $p^{m_0-1}(p-1)$ with respect to the sequence $(S(n,k))_{n\in\N_+}$ and prime number $p$ for which the value $l$ specified in the statement of Theorem \ref{thm1} is greater than $1$?
\end{prob}

We can show, that if $a<k<p$, then $l$ must be equal to $1$.

\begin{thm}
Let $k,a$ be positive integers and $p$ be a prime number such that $a<k<p$. Then for each positive integer $n\equiv a\pmod{p^{m_0-1}(p-1)}$ there holds $v_p(S(n,k))=v_p(S(a+p^{m_0-1}(p-1),k))+v_p(n-a)-m_0+1$.

In particular, for each positive integer $m\geq m_0$, $[a]_{p^{m-1}(p-1)}$ is the only non-constant class modulo $p^{m-1}(p-1)$ contained in $[a]_{p^{m_0-1}(p-1)}$ and its least $p$-adic valuation is equal to $v_p(S(a+p^{m_0-1}(p-1),k))+m-m_0$.
\end{thm}

\begin{proof}
We have $S(a,k)=0$. Since $k<p$, by Theorem 1, we have $\beta\in\Z$ and $l\in\N_+$ such that
\begin{equation}\label{form}
v_p(S(n,k))=\beta+lv_p(n-a)
\end{equation}
for $n\equiv a\pmod{p^{m_0-1}(p-1)}$. Putting $n=a+p^{m_0-1}(p-1)$ in (\ref{form}) we obtain $\beta=v_p(S(a+p^{m_0-1}(p-1),k))-l(m_0-1)$. It suffices to prove that $l=1$. We have $S(a+n(p-1),k)=\frac{1}{k!}f_{a,k}(n)$ for $n\in\N$, where $f_{a,k}(x)=\sum_{j=1}^k (-1)^{k-j}{k\choose j}j^a(j^{p-1})^x$, $x\in\Z_p$, is a $p$-adic locally analytic function. Since $f_{a,k}(0)=S(a,k)=0$ and $v_p(f_{a,k}(n))=v_p(S(a+n(p-1)))-v_p(k!)$, we thus infer that $l=\inf\{i\in\N: f_{a,k}^{(i)}\neq 0\}$. Hence it remains us to show that $f_{a,k}'(0)\neq 0$. Indeed,
$$f_{a,k}'(0)=\sum_{j=1}^k (-1)^{k-j}{k\choose j}j^a\log_p(j^{p-1})=\log_p\left(\prod_{j=1}^k j^{(-1)^{k-j}{k\choose j}j^a(p-1)}\right).$$
Let us note that for $k\geq 2$, by Tshebyshev theorem there exists a prime number $q$ such that $\frac{k}{2}<q\leq k$. Then the $q$-adic valuation of the number $\prod_{j=0}^k j^{(-1)^{k-j}{k\choose j}j^a(p-1)}$ is equal to $(-1)^{k-q}{k\choose q}q^a(p-1)\neq 0$. This means that $\prod_{j=0}^k j^{(-1)^{k-j}{k\choose j}j^a(p-1)}\not\in\{-1,1\}$ and from our preparation about the $p$-adic logarithm we conclude that $$f_{a,k}'(0)=\log_p\left(\prod_{j=0}^k j^{(-1)^{k-j}{k\choose j}j^a(p-1)}\right)\neq 0.$$
\end{proof}


\begin{thebibliography}{HD}






\normalsize
\baselineskip=17pt


\bibitem{AMM} T. Amdeberhan, D. Manna, V. Moll, {\it The 2-adic Valuation of Stirling Numbers}, Experiment. Math. (2008), 17:69-82.
\bibitem{BenDav} M. Bendersky, D.M. Davis, {\it 2-primary $v_1$-periodic homotopy groups of $SU(n)$}, Amer. J. Math. 114 (1991), 529-544.
\bibitem{BenMos} C. Bennett, E. Mosteig, {\it Congruence classes of 2-adic valuations of Stirling numbers of the second kind}, J. Integer Seq. 16 (3) (2013).
\bibitem{BeckRio} H. Becker, J. Riordan, {\it The arithmetic of Bell and Stirling numbers}, Amer. J. Math., 70:385-394, 1948.
\bibitem{BMMMN} A. Berrizbeitia, L. Medina, A. Moll, V. Moll, L. Noble, {\it The $p$-adic valuation of Stirling numbers}, Journal for Algebra and Number Theory Academia (2010), 1:1-30.
\bibitem{Car} L. Carlitz, {\it Congruences for generalized Bell and Stirling numbers}, Duke Math. J., 22:193-205, 1955.
\bibitem{Cla} F. Clarke, {\it Hensel's Lemma and the Divisibility by Primes of Stirling-like Numbers}, J. Numb. Theory vol. 52 (1) (1995), 69-84.
\bibitem{CraKna} M.C. Crabb, K. Knapp, {\it The Hurewicz map on stunted complex projective spaces}, Amer Jour Math 110 (1988), 783-809.
\bibitem{Dav1} D.M. Davis, {\it Divisibility by 2 of partial Stirling numbers}, Funct. Approx. Comment. Math. vol. 49 (1) (2013), 29-56.
\bibitem{Dav2} D.M. Davis, {\it Divisibility by 2 of Stirling-like numbers}, Proc Amer Math Soc 110 (1990), 597-600.
\bibitem{Dav3} D.M. Davis, {\it Divisibility by 2 and 3 of certain Stirling numbers}, Integers 8 (2008) A56.
\bibitem{Dav4} D.M. Davis, {\it $v_1$-periodic homotopy groups of $SU(n)$ at odd primes}, Proc. London Math. Soc. 43 (1991), 529-541.
\bibitem{DavPot} D.M. Davis and K. Potocka, {\it 2-primary $v_1$-periodic homotopy groups of $SU(n)$ revisited}, Forum Math 19 (2007), 783-822.
\bibitem{DavSun} D.M. Davis and Z.W. Sun, {\it A number-theoretic approach to homotopy exponents of $SU(n)$}, Jour Pure Appl Alg 209 (2007), 57-69.
\bibitem{GesLeng} I. M. Gessel, T. Lengyel, {\it On the order of Stirling numbers and alternating binomial coefficient sums}, Fibonacci Quart., 39 (5) (2001):444-454.
\bibitem{HZZ} S. Hong, J. Zhao, W. Zhao, {\it The 2-adic valuations of Stirling numbers of the second kind}, Int. J. Number Theory 8 (2012), no. 4, 1057-1066.
\bibitem{KomLam} G. Komawila, P. Lambrechts {\it Euler series, Stirling numbers and the growth of the homology of the space of long links}, Bull. Belg. Math. Soc. Simon Stevin vol. 20 (5) (2013), 843-857.
\bibitem{Le} A. M. Legendre, {\it Th\'{e}orie des nombres}, Firmin Didot Fr\`{e}res, Paris, 1830.
\bibitem{Leng} T. Lengyel, {\it On the divisibility by 2 of the Stirling numbers of the second kind}, The Fibonacci Quarterly 32, 1994, pp. 194-201.
\bibitem{Leng2} T. Lengyel, {\it On the 2-adic order of Stirling numbers of the second kind and their differences}, 21st International Conference on Formal Power Series and Algebraic Combinatorics (FPSAC 2009), Discete Math. Theor. Comput. Sci. Proc., AK, Assoc. Discete Math. Theor. Comput. Sci., Nancy, 2009, 561-572.
\bibitem{Lun1} A.T. Lundell, {\it A divisibility property for Stirling numbers}, Jour Number Theory 10 (1978), 35-54.
\bibitem{Lun2} A.T. Lundell, {\it Generalized $e$-invariants and the numbers of James}, Quar Jour Math Oxford 25 (1974), 427-440.
\bibitem{Mahl} K. Mahler, {\it $p$-adic Numbers and Their Functions}, second ed., Cambridge University Press, Cambidge, U.K., 1981.
\bibitem{Schi} W. H. Schikof, {\it Ultrametric calculus. An introduction to $p$-adic analysis}, Cambridge University Press, New York, USA, 1984.
\bibitem{Sko} Th. Skolem, {\it Einige Sätze über gewisse Reihenentwicklungen und exponentiale Beziehungen mit Anwendung auf diophantische Gleichungen}, Oslo Vid. akad. Skrifter, I (6), 1933.
\bibitem{Wan} S. De Wannemacker, {\it On 2-adic orders of Stirling numbers of the second kind}, Integers 5 (2005), A21.

\end{thebibliography}
\end{document}